\documentclass[11pt]{article}
\usepackage{amsfonts}
\usepackage{amsmath}
\usepackage{amsthm}
\usepackage{amssymb}
\usepackage{latexsym}
\usepackage[all]{xy}
\usepackage{color}

\theoremstyle{definition}
\newtheorem{thm}{Theorem}[section]
\newtheorem{prop}{Proposition}[section]

\newtheorem{lem}{Lemma}[section]
\newtheorem{defn}{Definition}[section]
\newtheorem{rem}{Remark}[section]
\newtheorem{case}{Case}

\newcommand{\al}{\alpha}

\newcommand{\id}{\mathrm{id}}

\newcommand{\str}{\mathrm{str}}

\newcommand{\mcB}{\mathcal{B}}

\newcommand{\mcF}{\mathcal{F}}
\newcommand{\mcH}{\mathcal{H}}
\newcommand{\mcI}{\mathcal{I}}

\newcommand{\mcN}{\mathcal{N}}

\newcommand{\mcV}{\mathcal{V}}
\newcommand{\mcW}{\mathcal{W}}

\newcommand{\mfe}{\mathfrak{e}}

\newcommand{\mfg}{\mathfrak{g}}
\newcommand{\mfgl}{\mathfrak{g}\mathfrak{l}}

\newcommand{\mfsl}{\mathfrak{s}\mathfrak{l}}

\newcommand{\mfst}{\mathfrak{st}}

\newcommand{\Ima}{\mathrm{Image}}

\newcommand{\Ker}{\mathrm{Ker}}

\newcommand{\wt}{\widetilde}

\newcommand{\C}{\mathbb{C}}

\newcommand{\Z}{\mathbb{Z}}

\newcommand{\ot}{\otimes}

\newcommand\ta{\tau}
\newcommand\be{\beta}
\def\Z{\mathbb Z} \def\ZZ{\Z}
\DeclareMathOperator{\HC}{HC}
\def\dirlim{\varinjlim}
\def\uce{\mathfrak{uce}}

\title{Universal central extensions of $\mfsl_{m|n}$ over $\Z/2\Z$-graded algebras}
\author{Hongjia Chen, Jie Sun}
\date{}

\begin{document}

\maketitle

\begin{abstract}
We study central extensions of the Lie superalgebra $\mfsl_{m|n}(A)$,
where $A$ is a $\Z/2\Z$-graded superalgebra over a commutative ring $K$.
The Steinberg Lie superalgebra $\mfst_{m|n}(A)$ plays a crucial role. We show that $\mfst_{m|n}(A)$ is a central extension of $\mfsl_{m|n}(A)$ for $m+n\geq 3$.
We use a $\Z/2\Z$-graded version of cyclic homology to show that the center
of the extension is isomorphic to $\HC_1(A)$ as $K$-modules. For $m+n\geq 5$, we prove that $\mfst_{m|n}(A)$ is the universal central extension of $\mfsl_{m|n}(A)$. For $m+n=3,4$, we prove that $\mfst_{2|1}(A)$ and $\mfst_{3|1}(A)$ are both centrally closed. The universal central extension of $\mfst_{2|2}(A)$ is constructed explicitly.\\


\end{abstract}

\section{Introduction}

Central extensions are important in physics as they can reduce the study of projective representations to the study of true representations. Although finite dimensional simple Lie algebras have no nontrivial central extensions, infinite dimensional Lie algebras, by contrast, have many interesting central extensions. In order to find ``all" central extensions of a given Lie algebra, it's often helpful to determine the universal central extension, which exists for perfect Lie algebras (well-known) and superalgebras \cite{Ne}. An example of infinite dimensional Lie algebra is the (untwisted) loop algebra $\mfg\otimes_{\C} \C[t^{\pm 1}]$, where $\mfg$ is a finite dimensional simple Lie algebra over $\C$. The universal central extension of $\mfg\otimes_{\C} \C[t^{\pm 1}]$, which is one dimensional, gives an affine Kac-Moody Lie algebra. The (untwisted) multiloop algebra $\mfg\otimes_{\C} \C[t_1^{\pm 1},t_2^{\pm 1},\cdots,t_n^{\pm 1}]$ and its universal central extension play an important role in the theory of extended affine Lie algebras \cite{AABGP}, \cite{Ne2}.

In \cite{Ka1} C. Kassel constructed the universal central extensions of $\mfg\otimes _K A$ by using K\"ahler differentials, where $\mfg$ is a finite dimensional simple Lie algebra over a commutative unital  ring $K$ and $A$ is a commutative, associative, unital $K$-algebra. Kassel's model was generalized in \cite{BK} under certain conditions. Twisted forms of $\mfg\otimes _K A$, which include twisted (multi)loop algebras as examples, and their central extensions were studied in \cite{S} by using the descent theory. In particular, the universal central extension of a multiloop Lie torus was given by the descent construction. A different construction of central extensions of centreless Lie tori by using centroidal derivations was given in \cite{Ne2}. When $\mfg$ is a basic classical Lie superalgebra over $K$, central extensions of $\mfg\otimes _K A$ were given in \cite{IoKo1} by using K\"ahler differentials. The second homology of the Lie superalgebra $\mfg\otimes _K A$ and its twisted version was calculated in \cite{IoKo2}.

When specializing $\mfg=\mfsl_n$ or $\mfsl_{m|n}$, central extensions of $\mfg\otimes _K A$ for more general $A$ have been studied by several authors. In \cite{KaLo}, the universal central extension of $\mfsl_n(A)$  for $n\geq 5$, where $A$ is an associative, unital $K$-algebra, was constructed by cyclic homology of $A$. The universal central extension of $\mfsl_{m|n}(A)$ was studied in \cite{MiPi1}  for $m+n\geq 5$, and in \cite{SCG} for $m+n=3,4$. When $A$ is a $\Z/2\Z$-graded, associative, unital $K$-superalgebra, the universal central extension of $\mfsl_n(A)$ was constructed by a $\Z/2\Z$-graded version of cyclic homology of $A$ in \cite{CG}. Thus it's natural to consider the matrix Lie superalgebra $\mfsl_{m|n}(A)$, where $A$ is a $\Z/2\Z$-graded, associative, unital $K$-superalgebra, since it combines the two cases of $\mfsl_{m|n}(A)$ for $A$ a (non-super)algebra and $\mfsl_{n}(A)$ for $A$ a superalgebra. In the case where $A$ is supercommutative, $\mfsl_{m|n}(A)$ and its central extensions have applications to superconformal field theory and superstring theory (see \cite{KT}).

In this paper, after reviewing some basics of central extensions of Lie superalgebras, we define the Lie superalgebra $\mfsl_{m|n}(A)$, for $m+n\geq 3$, to be the subsuperalgebra of $\mfgl_{m|n}(A)$ generated by the elements
$E_{ij}(a)$, $1 \leq i \neq j \leq m+n$, $a \in A$, where $E_{ij}(a)$ has entry $a$ at the position $(ij)$ and $0$ elsewhere. We show that $\mfsl_{m|n}(A)=[\mfgl_{m|n}(A),\mfgl_{m|n}(A)]$. Under the usual supertrace of a matrix in $\mfgl_{m|n}(A)$, not every matrix in $\mfsl_{m|n}(A)$ has supertrace in $[A,A]$. Thus we introduce a new supertrace on $\mfgl_{m|n}(A)$ by defining $\mathrm{str}(E_{ij}(a)) : = \delta_{ij} (-1)^{|i|(|i|+|a|)}a$. Under this new supertrace, every matrix in $\mfsl_{m|n}(A)$ has supertrace in $[A,A]$. After showing the perfectness of $\mfsl_{m|n}(A)$, we study central extensions of $\mfsl_{m|n}(A)$ by defining the Steinberg Lie superalgebra $\mfst_{m|n}(A)$. We show that $\mfst_{m|n}(A)$ is a central extension of $\mfsl_{m|n}(A)$ for $m+n\geq 3$, and the kernel of this central extension is isomorphic to the first $\Z/2\Z$-graded cyclic homology group of $A$. For $m+n\geq 5$, we prove that $\mfst_{m|n}(A)$ is the universal central extension of $\mfsl_{m|n}(A)$. As an application, we describe the universal central extension of the Lie superalgebra $\mfsl_I(A)$, where $I=I_{\bar{0}}\cup I_{\bar{1}}$ is a (possibly infinite) superset with $|I|\geq 5$. For $m+n=3,4$, we prove that $\mfst_{2|1}(A)$ and $\mfst_{3|1}(A)$ are both centrally closed. The universal central extension of $\mfst_{2|2}(A)$ is constructed explicitly through a super $2$-cocycle.

\section{Review of central extensions of Lie superalgebras}

Let $K$ be a commutative unital ring. We assume throughout this paper
that $2$ is invertible in $K$. A \textit{$K$-superalgebra} is a $K$-supermodule $A=A_{\bar{0}}\oplus A_{\bar{1}}$ together with a $K$-bilinear map $A\times A\rightarrow A$, called \textit{product}, of degree $0$. A $K$-superalgebra $L$ with product $[\cdot, \cdot]$ is a \textit{Lie $K$-superalgebra} if for all $x,y,z\in L$, we have $$[x,y]=-(-1)^{|x||y|}[y,x],$$
$$[x,[y,z]]=[[x,y],z]+(-1)^{|x||y|}[y,[x,z]].$$
Formulas involving degree functions are supposed to be valid for homogeneous elements.

Let $L=L_{\bar{0}}\oplus L_{\bar{1}}$ be a Lie superalgebra over $K$. The pair
$(\wt{L},\varphi)$, where $\wt{L}=\wt{L}_{\bar{0}}\oplus \wt{L}_{\bar{1}}$ is
a Lie superalgebra and $\varphi: \wt{L} \rightarrow L$ an epimorphism,
is called a \textit{central extension} of $L$ if $\Ker(\varphi)\subset Z(\wt{L})$,
where $Z(\wt{L}):=\{z\in \wt{L}\ |\  [z,x]=0$ for all $x\in \wt{L}\}$ is the \textit{center} of $\wt{L}$.
Thus $[\Ker(\varphi), \wt{L}]=0$.
A central extension $(\wt{L}, \varphi)$ of $L$ is called \textit{universal}
if 
for any
central extension $(L', \psi)$ of $L$ there exists a unique
homomorphism $\eta: \wt{L} \rightarrow L'$ such that $\psi \circ \eta
=\varphi$. A universal central extension of $L$ exists and is then unique up to a unique isomorphism if and only if $L$ is perfect \cite{Ne}.

Similar to Lie algebras, a central extension of a Lie
$K$-superalgebra $L$ can be constructed by using a $2$-cocycle $\ta
: L \times L \to C$. Here $C$ is a $K$-supermodule, $\ta$ is
$K$-bilinear of degree $0$ whence $\ta(L_\al, L_\be) \subset
C_{\al+\be}$ for $\al,\be \in \ZZ/2\ZZ$, alternating in the sense
that $\ta(x, y) + (-1)^{|x||y|} \ta (y,x) = 0$ and satisfies
$$(-1)^{|x||z|} \ta(x, [y,z]) +(-1)^{|y||x|}\ta(y, [z,x])
+(-1)^{|z||y|}\ta(z, [x,y])=0.$$ Equivalently, a $2$-cocycle is a
map $\ta : L \times L \to C$ such that $L\oplus C$ is a Lie
superalgebra with respect to the grading $(L\oplus C)_\al = L_\al
\oplus C_\al$ and product $[l_1\oplus c_1, l_2\oplus
c_2]=[l_1,l_2]_L\oplus \tau(l_1,l_2)$ where $[.,.]_L$ is the product
of $L$. In this case, the canonical projection $L\oplus C \to L$ is
a central extension.

\section{The Lie superalgebra $\mfsl_{m|n}(A)$}

Let $A=A_{\bar{0}} \oplus A_{\bar{1}}$ be a $\Z/2\Z$-graded, associative, unital
$K$-superalgebra. Elements of $K$ have degree $0$. Let $M_{m|n}(A)$ be the $(m+n) \times (m+n)$ matrix superalgebra with
coefficients in $A$ and $\deg (E_{ij}(a)) =\deg(E_{ij})+\deg (a)
= |i|+|j|+|a|$, for $1 \leq i,j \leq m+n$ and any homogeneous element
$a \in A$, where $|i| = 0$ if $1 \leq i \leq m$, and $|i|=1$ otherwise. Here $E_{ij}(a)$ has entry $a$ at the position $(ij)$ and $0$ elsewhere. Let $\mfgl_{m|n}(A)$ be the Lie superalgebra associated to the associative superalgebra $M_{m|n}(A)$. Its product is given by $[X,Y]=XY-(-1)^{|X||Y|}YX$.

\begin{defn} For $m+n \geq 3$, the Lie superalgebra $\mfsl_{m|n}(A)$ is the
subsuperalgebra of $\mfgl_{m|n}(A)$ generated by the elements
$E_{ij}(a)$, $1 \leq i \neq j \leq m+n$, $a \in A$.
\end{defn}

The canonical matrix units $E_{ij}(a), E_{kl}(b)\in \mfsl_{m|n}(A)$ satisfy the relation $$[E_{ij}(a),E_{kl}(b)]=\delta_{jk}E_{il}(ab)-(-1)^{|E_{ij}(a)||E_{kl}(b)|}\delta_{li}E_{kj}(ba).$$ The following lemma shows that the Lie superalgebra $\mfsl_{m|n}(A)$ can be equivalently defined as $\mfsl_{m|n}(A)=[\mfgl_{m|n}(A),\mfgl_{m|n}(A)]$, the derived subsuperalgebra of
$\mfgl_{m|n}(A)$.

\begin{lem} The Lie superalgebra $\mfsl_{m|n}(A)$ satisfies $\mfsl_{m|n}(A)=[\mfgl_{m|n}(A),\mfgl_{m|n}(A)]$.
\end{lem}

\begin{proof}
Since $E_{ij}(a)=[E_{ik}(a), E_{kj}(1)]$ for distinct $i,j,k$, we have $\mfsl_{m|n}(A) \subseteq [\mfgl_{m|n}(A),\mfgl_{m|n}(A)]$. On the other hand, for $i\neq j$, we have
\begin{equation*}
\begin{split}
&[E_{ij}(a), E_{ji}(b)] - (-1)^{|a|\cdot |b|} [E_{ij}(1),E_{ji}(ba)]  \\
= &E_{ii}(ab) - (-1)^{(|i|+|j|+|a|)(|i|+|j|+|b|)}E_{jj}(ba) - (-1)^{|a|\cdot |b|}
(E_{ii}(ba) - (-1)^{(|i|+|j|)(|i|+|j|+|a|+|b|)}E_{jj}(ba) )\\
= & E_{ii}(ab) - (-1)^{|a|\cdot |b|} E_{ii}(ba) = E_{ii}([a,b]) \\
= & [E_{ii}(a),E_{ii}(b)].
\end{split}
\end{equation*}
Thus $[\mfgl_{m|n}(A),\mfgl_{m|n}(A)] \subseteq \mfsl_{m|n}(A)$. Hence
$\mfsl_{m|n}(A)=[\mfgl_{m|n}(A),\mfgl_{m|n}(A)]$.
\end{proof}

The usual supertrace of a matrix $X=(x_{ij})\in \mfgl_{m|n}(A)$ is given by $\sum_{i=1}^m x_{ii}-\sum_{i=m+1}^{m+n}x_{ii}$. Under this supertrace, not every matrix in $\mfsl_{m|n}(A)$ has supertrace in $[A,A]$, where $[A,A]$ is the span of all commutators $[a_1, a_2]=a_1 a_2-(-1)^{|a_1||a_2|} a_2 a_1$, $a_i\in A$. Thus we introduce a new supertrace on $\mfgl_{m|n}(A)$. Define
$\mathrm{str}(E_{ij}(a)) : = \delta_{ij} (-1)^{|i|(|i|+|a|)}a$ and extend
linearly to all of $\mfgl_{m|n}(A)$. This supertrace is equivalent to the definition in \cite{M}. If $n=0$ or $A_{\bar{1}}=\{0\}$, then this new supertrace coincides with the usual supertrace. Using this new supertrace, we have the following lemma.

\begin{lem} The Lie superalgebra $\mfsl_{m|n}(A)$ satisfies $\mfsl_{m|n}(A) \subset \{X \in \mfgl_{m|n}(A) \,|\, \mathrm{str}(X) \in [A,A]\}$. In addition, if $m\geq 1$, then we have $\mfsl_{m|n}(A) = \{X \in \mfgl_{m|n}(A) \,|\, \mathrm{str}(X) \in [A,A]\}$.
\end{lem}

\begin{proof} Denote $\overline{\mfsl}_{m|n}(A) = \{X \in \mfgl_{m|n}(A) \,|\, \mathrm{str}(X) \in [A,A]\}$.
For any $1 \leq i,j,k,l \leq m+n$ and homogeneous elements $a, b \in A$, we have
\begin{equation*}
\begin{split}
\str\big([E_{ij}(a), E_{kl}(b)]\big) = & \str\Big(\delta_{jk}E_{il}(ab)
- (-1)^{(|i|+|j|+|a|)(|k|+|l|+|b|)}\delta_{il}E_{kj}(ba)\Big)  \\
= & \delta_{il}\delta_{jk} \Big( (-1)^{|i|(|i|+|a|+|b|)}ab -
(-1)^{(|i|+|j|+|a|)(|j|+|i|+|b|)}(-1)^{|j|(|j|+|a|+|b|)}ba\Big) \\
= & \delta_{il}\delta_{jk} (-1)^{|i|(|i|+|a|+|b|)}\big(ab -(-1)^{|a| \cdot |b|}ba\big)
\in [A,A].
\end{split}
\end{equation*}
Thus $\mfsl_{m|n}(A) \subseteq \overline{\mfsl}_{m|n}(A)$. In addition, if $m\geq 1$, let
$X = \sum\limits_{i=1}^{m+n} E_{ii}(a_i) \in \overline{\mfsl}_{m|n}(A)$, where $a_i$'s are homogeneous
elements in $A$. Then $\str(X) = \sum\limits_{i=1}^{m+n} (-1)^{|i|(|i|+|a_i|)}a_i \in [A,A]$.
Moreover, we have
\begin{equation*}
\begin{split}
& X+ \sum_{i=2}^{m+n}(-1)^{|i|(|i|+|a_i|)}[E_{1i}(a_i),E_{i1}(1)] =
\sum\limits_{i=1}^{m+n} E_{ii}(a_i) + \sum_{i=2}^{m+n}
\Big((-1)^{|i|(|i|+|a_i|)}E_{11}(a_i) - E_{ii}(a_i) \Big) \\
= & E_{11}(\sum\limits_{i=1}^{m+n} (-1)^{|i|(|i|+|a_i|)}a_i) \in E_{11}([A,A])
= [E_{11}(A),E_{11}(A)] \subseteq [\mfgl_{m|n}(A),\mfgl_{m|n}(A)]=\mfsl_{m|n}(A).
\end{split}
\end{equation*}
Thus $X\in \mfsl_{m|n}(A)$ and $\overline{\mfsl}_{m|n}(A) \subseteq \mfsl_{m|n}(A)$. Hence $\mfsl_{m|n}(A)=\overline{\mfsl}_{m|n}(A)$.
\end{proof}

\begin{lem} The Lie superalgebra $\mfsl_{m|n}(A)$ is perfect, i.e.,
$[\mfsl_{m|n}(A),\mfsl_{m|n}(A)]=\mfsl_{m|n}(A)$.
\end{lem}
\begin{proof} The perfectness follows from
$E_{ij}(a) = [E_{ik}(a),E_{kj}(1)]$ for distinct $i,j,k$ and any $a \in A$.
\end{proof}

The perfectness of $\mfsl_{m|n}(A)$ guarantees that the universal central extension of $\mfsl_{m|n}(A)$ exists.

\section{Central extensions of $\mfsl_{m|n}(A)$}

To study central extensions of $\mfsl_{m|n}(A)$, we define the Steinberg Lie superalgebra $\mfst_{m|n}(A)$ as follows by generators and relations.

\begin{defn}
For $m+n \geq 3$, the Steinberg Lie superalgebra $\mfst_{m|n}(A)$ is defined
to be the Lie superalgebra over $K$ generated by the homogeneous
elements $F_{ij}(a)$, $a \in A$ homogeneous, $1\leq i \neq j\leq m+n$
and $\deg (F_{ij}(a))=|i|+|j|+|a|$, subject to the following relations for $a,b
\in A$:
\begin{align}
&a \mapsto F_{ij}(a) \text{ is a $K$-linear map,} \label{st1}\\
&[F_{ij}(a), F_{jk}(b)] = F_{ik}(ab), \text{ for distinct } i, j, k,
\label{st2} \\
&[F_{ij}(a), F_{kl}(b)] = 0, \text{ for } i\neq j\neq k \neq l \neq i.
\label{st3}
\end{align}
\end{defn}

Let $\mcN^+$ and $\mcN^-$ be the $K$-submodules of
$\mfst_{m|n}(A)$ generated by $F_{ij}(a)$ for $1\le i < j \le m+n$ and $F_{ij}(a)$ for
$1\le j < i \le m+n$, respectively. It is clear that $\mcN^+$ and $\mcN^-$ are subsuperalgebras of $\mfst_{m|n}(A)$. Let $\mcH$ be the $K$-submodule of $\mfst_{m|n}(A)$ generated by $H_{ij}(a,b):=[F_{ij}(a), F_{ji}(b)]$ for $1\le i \neq j \le m+n$. Clearly, $H_{ij}(a,b)= -(-1)^{(|i|+|j|+|a|)(|i|+|j|+|b|)}H_{ji}(b,a)$. The following lemma implies that $\mcH$ is indeed a subsuperalgebra of $\mfst_{m|n}(A)$.

\begin{lem}\label{HN} For distinct $i,j,k$, we have the following identities.
\begin{enumerate}
\item
$[H_{ij}(a, b), F_{ik}(c)]= F_{ik}(abc)$.

\item
$[H_{ij}(a, b), F_{ki}(c)] =  -(-1)^{(|a|+|b|)(|i|+|k|+|c|)}F_{ki}(cab)$. \label{HF1}

\item
$[H_{ij}(a, b), F_{ij}(c)]= F_{ij}\Big(abc +(-1)^{(|i|+|j|+|a||b|+|b||c|+|c||a|)}cba\Big)$.\label{HF2}
\end{enumerate}
\end{lem}

\begin{proof} Since $H_{ij}(a,b)=[F_{ij}(a), F_{ji}(b)]$, by the super Jacobi identity we have
\begin{equation*}
\begin{split}
[H_{ij}(a, b), F_{ik}(c)]= & [[F_{ij}(a),F_{ji}(b)], F_{ik}(c)]
= [F_{ij}(a),[F_{ji}(b), F_{ik}(c)]] +0 \\
= & [F_{ij}(a),F_{jk}(bc)] = F_{ik}(abc), \text{ and}
\end{split}
\end{equation*}
\begin{equation*}
\begin{split}
[H_{ij}(a, b), F_{ki}(c)]= & [[F_{ij}(a),F_{ji}(b)], F_{ki}(c)]
= (-1)^{(|i|+|j|+|b|)(|i|+|k|+|c|)}[[F_{ij}(a), F_{ki}(c)],F_{ji}(b)] + 0\\
= & -(-1)^{(|i|+|j|+|b|)(|i|+|k|+|c|)+(|i|+|j|+|a|)(|i|+|k|+|c|)}[F_{kj}(ca),F_{ji}(b)] \\
= & -(-1)^{(|a|+|b|)(|i|+|k|+|c|)}F_{ki}(cab).
\end{split}
\end{equation*}
Apply the first two identities, and notice that $\deg(H_{ij}(a,b))=|a|+|b|$, we have
\begin{equation*}
\begin{split}
[H_{ij}(a, b),F_{ij}(c)]= & [H_{ij}(a, b), [F_{ik}(c),F_{kj}(1)]] \\
= & [[H_{ij}(a, b), F_{ik}(c)],F_{kj}(1)] + (-1)^{(|a|+|b|)(|i|+|k|+|c|)}[F_{ik}(c),[H_{ij}(a, b),F_{kj}(1)]] \\
= & [F_{ik}(abc),F_{kj}(1)] -(-1)^{(|i|+|j|+|a||b|)+(|a|+|b|)(|j|+|k|+|c|)}[F_{ik}(c),[H_{ji}(b,a),F_{kj}(1)]] \\
= & F_{ij}(abc) +(-1)^{(|a|+|b|)(|j|+|k|)}(-1)^{(|i|+|j|+|a||b|)+(|a|+|b|)(|j|+|k|+|c|)}[F_{ik}(c),F_{kj}(ba)] \\
= & F_{ij}(abc +(-1)^{(|i|+|j|+|a||b|+|b||c|+|c||a|)}cba). \qedhere
\end{split}
\end{equation*}
\end{proof}

From the above lemma we have $[\mcH,\mcN^\pm] \subseteq \mcN^\pm$, thus $\mcN^+ + \mcH + \mcN^-$ is an ideal of $\mfst_{m|n}(A)$ and contains all generators, hence
$\mfst_{m|n}(A) = \mcN^+ + \mcH + \mcN^-$.
Let $\varphi$ be the canonical Lie superalgebras epimorphism
$$\varphi:  \mfst_{m|n}(A)  \to \mfsl_{m|n}(A), \ \quad F_{ij}(a) \mapsto E_{ij}(a).$$ Then $\varphi|_{\mcN^+}$ and $\varphi|_{\mcN^-}$ are injective since $\mcN^+ = \oplus_{1\le i < j \le m+n}F_{ij}(A)$ and
$\mcN^- = \oplus_{1\le j < i \le m+n}F_{ij}(A)$, where $F_{ij}(A)$ is the $K$-submodule $\{F_{ij}(a)\,|\, a\in A\}$. We claim that the sum $ \mcN^+ + \mcH + \mcN^-$ is a direct sum. Indeed, if $n_1+h+n_2=0$ for $n_1 \in \mcN^+$,
$h \in \mcH$ and $n_2 \in \mcN^-$, then $0 = \varphi(n_1+h+n_2) =\varphi(n_1)+ \varphi(h)+\varphi(n_2)$.
Since $\varphi(n_1)$ is the uppertriangular part of $\varphi(n_1+h+n_2)$ in $\mfsl_{m|n}(A)$ , thus $\varphi(n_1)=0$.
The injectivity of $\varphi|_{\mcN^+}$ implies $n_1=0$.
Similarly $n_2=0$, and hence $h=0$. Thus we have the direct sum $\mfst_{m|n}(A) =\mcN^+ \oplus \mcH \oplus \mcN^-$,
in other word, we have
\begin{equation} \label{direct}
\mfst_{m|n}(A) = \mcH \bigoplus \bigg(\sum\limits_{1 \le i \neq j \le m+n}F_{ij}(A)\bigg).
\end{equation}

\begin{prop}\label{inH}
For $m+n \geq 3$, $(\mfst_{m|n}(A),
\varphi)$ is a central extension of
$\mfsl_{m|n}(A)$.
\end{prop}
\begin{proof}
Since $\mfst_{m|n}(A) =\mcN^+ \oplus \mcH \oplus \mcN^-$, the injectivities of
$\varphi|_{\mcN^+}$ and $\varphi|_{\mcN^-}$ implies $\Ker (\varphi) \subseteq \mcH$. By Lemma \ref{HN}, we have
$[\mcH,\mcN^+ \oplus \mcN^-] \subseteq \mcN^+ \oplus \mcN^-$. For $t \in \Ker (\varphi)$, $F_{ij}(a)\in \mfst_{m|n}(A)$,
we have $[t, F_{ij}(a)] = n_1+n_2$ for some $n_1 \in \mcN^+$ and $n_2 \in \mcN^-$. This implies that
\begin{equation*}
\varphi(n_1+n_2) = \varphi\Big([t,F_{ij}(a)]\Big) = [\varphi(t),\varphi(F_{ij}(a))] =0,
\end{equation*}
hence $n_1=n_2=0$. Thus $\Ker (\varphi)$ is in the center
of the Lie superalgebra $\mfst_{m|n}(A)$.
\end{proof}

\section{Cyclic homology of $\Z/2\Z$-graded algebras}

In this section, we will recall the cyclic homology of $\Z/2\Z$-graded algebras and relate it to the kernel of the central extension $(\mfst_{m|n}(A), \varphi)$ of
the Lie superalgebra $\mfsl_{m|n}(A)$.

Cyclic homology of $\Z/2\Z$-graded algebras was studied in \cite{Ka2} (see also \cite{IoKo2}). Define the chain complex of $K$-modules $C_*(A)$ by letting
$C_0(A)=A$ and, for $n \geq 1$, let the module $C_n(A)$ be the quotient of
the $K$-module $A^{\otimes (n+1)}$ by the $K$-submodule
$I_n$ generated by the elements
$$
a_0 \otimes a_1 \otimes \cdots \otimes a_n - (-1)^{n+|a_n|
\sum\limits_{i=0}^{n-1}|a_i|} a_n \otimes a_0 \otimes a_1 \otimes
\cdots \otimes a_{n-1},
$$
with homogeneous elements $a_i \in A$ for $0 \leq i \leq n$. The homomorphism
$\widetilde{d}_n: A^{\otimes (n+1)} \rightarrow A^{\otimes n}$ is
given by
\begin{eqnarray*}
\widetilde{d}_n(a_0 \otimes a_1 \otimes \cdots \otimes a_n) &=&
\sum_{i=0}^{n-1} (-1)^i a_0 \otimes \cdots \ot a_i a_{i+1} \otimes
\cdots \otimes a_n\\
&&+ (-1)^{n+|a_n| \sum\limits_{i=0}^{n-1}|a_i|} a_n
a_0 \otimes a_1 \otimes \cdots \otimes a_{n-1}.   \label{tchg1}
\end{eqnarray*}
One can check $\widetilde{d}_n (I_n) \subseteq I_{n-1}$, hence
it induces a homomorphism $d_n:C_n(A) \rightarrow C_{n-1}(A)$ and
we have $d_{n-1}d_{n}=0$. The $n^{\mathrm{th}}$ $\Z/2\Z$-graded cyclic
homology group $\HC_n(A)$ of the superalgebra $A$ is defined as
$\HC_n(A)=\Ker(d_n)/\Ima(d_{n+1})$.

We will see that the first cyclic homology group $\HC_1(A)$ is closely related to $\Ker(\varphi)$. The $K$-module $\HC_1(A)$ can be written down explicitly as following.
Let $$
      \ll A , A\gg \; = (A\ot_K A) / \mathcal{I},
$$ where $\mathcal{I}$ is the $K$-span of all elements of type
\begin{align*}
  a &\otimes b+ (-1)^{|a| |b|}  b\otimes a,  \\
 (-1)^{|a||c|} a &\otimes bc +(-1)^{|b||a|} b\otimes
          ca +(-1)^{|c||b|} c\otimes ab
\end{align*}
for $a,b,c\in A$. We abbreviate $\ll a,b \gg \, = a\ot b +
\mathcal{I}$. Observe that there is a well-defined commutator map
$$
 d: \ \ll A, A\gg \; \to A, \quad \ll a,b\gg \,\mapsto [a,b].
$$
Then
$$
  \HC_1(A)  = \{ \textstyle \sum_i \ll a_i, b_i \gg
   \,| \,  \sum_i [a_i, b_i] = 0 \}$$

We have shown that $\mfst_{m|n}(A) =\mcN^+ \oplus \mcH \oplus \mcN^-$. The injectivities of
$\varphi|_{\mcN^+}$ and $\varphi|_{\mcN^-}$ implies $\Ker (\varphi) \subseteq \mcH$. For the rest of this section we will assume $m\geq 1$. Consider the elements $$h(a,b):= H_{1j}(a,b)-(-1)^{|a||b|}H_{1j}(1,ba)$$ in $\mcH$. The following lemma shows that the definition of $h(a,b)$ does
not depend on $j$, for $j\neq 1$.

\begin{lem}\label{Hidentities}
For homogeneous elements $a,b,c \in A$, we have the following identities.

(1). $H_{ij}(ab,c) = H_{ik}(a,bc) - (-1)^{(|i|+|j|+|a|+|b|)(|i|+|j|+|c|)}H_{jk}(ca,b)$ for distinct $i,j,k$.

(2). $H_{1j}(a,b)-(-1)^{|a||b|}H_{1j}(1,ba)= H_{1k}(a,b)-(-1)^{|a||b|} H_{1k}(1,ba)$ for any $j,k\neq 1$.

(3). $H_{ij}(1,a)=-(-1)^{(|i|+|j|+|a|)(|i|+|j|)}H_{ji}(1,a)=H_{ij}(a,1)= -(-1)^{(|i|+|j|+|a|)(|i|+|j|)}H_{ji}(a,1)$.
\end{lem}
\begin{proof} Since $H_{ij}(ab,c) = [F_{ij}(ab),F_{ji}(c)]$ and $F_{ij}(ab)=[F_{ik}(a),F_{kj}(b)]$ for $k\neq i,j$, we have
\begin{equation*}
\begin{split}
H_{ij}(ab,c) = & [F_{ij}(ab),F_{ji}(c)] = [[F_{ik}(a),F_{kj}(b)],F_{ji}(c)] \\
= & [F_{ik}(a),[F_{kj}(b),F_{ji}(c)]] + (-1)^{(|k|+|j|+|b|)(|i|+|j|+|c|)}[[F_{ik}(a),F_{ji}(c)],F_{kj}(b)] \\
= & [F_{ik}(a),F_{ki}(bc)]] - (-1)^{(|k|+|j|+|b|+|k|+|i|+|a|)(|i|+|j|+|c|)}[[F_{ji}(c),F_{ik}(a)],F_{kj}(b)] \\
= & H_{ik}(a,bc) - (-1)^{(|i|+|j|+|a|+|b|)(|i|+|j|+|c|)}[F_{jk}(ca),F_{kj}(b)] \\
= & H_{ik}(a,bc) - (-1)^{(|i|+|j|+|a|+|b|)(|i|+|j|+|c|)}H_{jk}(ca,b), \text{ and}
\end{split}
\end{equation*}
\begin{equation*}
\begin{split}
H_{1j}(a,b) -  (-1)^{|a||b|} H_{1j}(1,ba) = &H_{1j}(a \cdot 1,b) - (-1)^{|a||b|} H_{1j}(1 \cdot 1,ba) \\
= & H_{1k}(a ,b) - (-1)^{(|1|+|j|+|a|)(|1|+|j|+|b|)}H_{jk}(ba,1)  \\
  \qquad & - (-1)^{|a||b|} \Big(H_{1k}(1 ,ba) - (-1)^{(|1|+|j|)(|1|+|j|+|a|+|b|)}H_{jk}(ba,1)\Big)  \\
= & H_{1k}(a,b)-(-1)^{|a||b|} H_{1k}(1,ba).
\end{split}
\end{equation*}
If we take $a=b=1$ in (1) and switch $j$ and $k$, we can conclude (3).
\end{proof}

The elements $h(a,b)$ in $\mcH$ satisfy nice identities, as shown in the following lemma, which suggests key connections to the elements $\ll a , b\gg$ in $\ll A , A\gg$.

\begin{lem}\label{habc} The elements $h(a,b)$ satisfy the following identities.

(1). $h(a,1)=h(1,b)=0$.

(2). $(-1)^{|a||c|}h(a,bc) + (-1)^{|b||a|} h(b,ca)+(-1)^{|c||b|}h(c,ab)=0$.

(3). $h(a,b) + (-1)^{|a||b|} h(b,a)=0$.

\end{lem}

\begin{proof} The identity (1) follows from Lemma \ref{Hidentities} (3). To prove the identity (2), we notice that
\begin{equation*}
\begin{split}
h(ab,c)=& H_{1j}(ab,c) - (-1)^{|c|(|a|+|b|)}H_{1j}(1,cab) \\
= &  H_{1k}(a,bc) - (-1)^{(|j|+|a|+|b|)(|j|+|c|)} H_{jk}(ca,b)  \\
 \qquad & - (-1)^{|c|(|a|+|b|)} \Big( H_{1k}(1,cab) - (-1)^{|j|(|j|+|a|+|b|+|c|)} H_{jk}(cab,1) \Big) \\
= & h(a,bc)  - (-1)^{(|j|+|a|+|b|)(|j|+|c|)} H_{jk}(ca \cdot 1,b) + (-1)^{(|j|+|a|+|b|)(|j|+|c|)} H_{jk}(cab \cdot 1,1)  \\
 \qquad & - (-1)^{|c|(|a|+|b|)} H_{1k}(1,cab) + (-1)^{|a|(|b|+|c|)} H_{1k}(1,bca)  \\
= & h(a,bc)  - (-1)^{(|j|+|a|+|b|)(|j|+|c|)} (-1)^{(|j|+|b|)(|j|+|a|+|c|)}\Big(H_{1k}(1,bca) - H_{1j}(b,ca)\Big) \\
 \quad & + (-1)^{(|j|+|a|+|b|)(|j|+|c|)} (-1)^{|j|(|j|+|a|+|b|+|c|)}\Big(H_{1k}(1,cab ) - H_{1j}(1,cab )\Big) \\
 \qquad & - (-1)^{|c|(|a|+|b|)} H_{1k}(1,cab) + (-1)^{|a|(|b|+|c|)} H_{1k}(1,bca)  \\
= & h(a,bc) + (-1)^{|a|(|b|+|c|)} H_{1j}(b,ca) -(-1)^{|c|(|a|+|b|)} H_{1j}(1,cab )\\
= & h(a,bc) +(-1)^{|a|(|b|+|c|)} h(b,ca).
\end{split}
\end{equation*}
Thus we have $$(-1)^{|a||c|}h(ab,c) = (-1)^{|a||c|}h(a,bc) + (-1)^{|b||a|} h(b,ca).$$ Since $(-1)^{|a||c|}h(ab,c)=-(-1)^{|c||b|}h(c, ab)$, we have the identity (2). If we let $c=1$ in (2) and use the fact $h(1, ab)=0$, we can conclude (3).
\end{proof}

The above lemma guarantees a well-defined $K$-module homomorphism $$\mu:\ \ll A, A\gg \; \to \mfst_{m|n}(A), \, \quad \ll a,b\gg \,\mapsto h(a,b).$$ We claim that $\mu(\HC_1(A))\subset \Ker(\varphi)$. Indeed, if $\sum_{i}\ll a_i,b_i\gg\,\in \HC_1(A)$, then $\sum_i [a_i, b_i]=0$. Since $\mu(\sum_{i}\ll a_i,b_i\gg)=\sum_i h(a_i, b_i)$ and
\begin{eqnarray*}
\varphi\big(\sum_i h(a_i, b_i)\big)&=&\varphi\big(\sum_i H_{1j}(a_i, b_i)-(-1)^{|a_i||b_i|}H_{1j}(1, b_i a_i)\big)\\
&=&\sum_i [E_{1j}(a_i), E_{j1}(b_i)]-(-1)^{|a_i||b_i|}[E_{1j}(1), E_{j1}(b_i a_i)]\\
&=& \sum_i \big(E_{11}(a_i b_i)-(-1)^{(|j|+|a_i|)(|j|+|b_i|)}E_{jj}(b_i a_i)\big)\\
&\quad&-(-1)^{|a_i||b_i|}\big( E_{11}(b_i a_i)-(-1)^{|j|(|j|+|a_i|+|b_i|)}E_{jj}(b_i a_i)\big)\\
&=&\sum_i E_{11}(a_i b_i-(-1)^{|a_i||b_i|}b_i a_i)\\
&=&E_{11} (\sum _i [a_i, b_i])=0.
\end{eqnarray*}
Thus $\mu(\HC_1(A))\subset \Ker(\varphi)$. Therefore the restriction of $\mu$ to $\HC_1(A)$ gives a $K$-module homomorphism from $\HC_1(A)$ to $\Ker (\varphi)$, namely
$$\mu|_{\HC_1(A)}:\ \HC_1(A) \, \to\; \Ker(\varphi), \, \quad \ll a,b\gg \,\mapsto h(a,b).$$
In order to show that $\mu|_{\HC_1(A)}$ is surjective, we will need the following lemma. Recall that $\mcH$ is the $K$-submodule of $\mfst_{m|n}(A)$ generated by $H_{ij}(a,b)$ for $1\le i \neq j \le m+n$.

\begin{lem}\label{elementinH}
Every element $x \in \mcH$ can be written as $x = \sum\limits_{i \in
I_x}h(a_i,b_i)+ \sum\limits_{j=2}^{m+n} H_{1j}(1,c_j)$,
where $a_i,b_i,c_j \in A$ and $I_x$ is a finite index set.
\end{lem}
\begin{proof}
Let $H_{ij}(a,b) \in \mcH$, where $1\le i \neq j \le m+n$. We consider the following two cases.
\begin{case}
$1 \in \{i,j\}$. Without loss of generality, we may assume $i=1$. In this case, we have
\begin{equation*}
H_{1j}(a,b) = H_{1j}(a,b)-(-1)^{|a||b|}H_{1j}(1,ba) + (-1)^{|a||b|}H_{1j}(1,ba)
= h(a,b) + (-1)^{|a||b|}H_{1j}(1,ba).
\end{equation*}
\end{case}
\begin{case}
$1 \notin \{i,j\}$. By Lemma \ref{Hidentities} (1), we have
\begin{equation*}
H_{ij}(a,b) = H_{i1}(a,b)-(-1)^{(|i|+|j|+|a|)(|i|+|j|+|b|)}H_{j1}(ba,1).
\end{equation*}
Thus we reduce this case to the Case 1. \qedhere
\end{case}
\end{proof}

\begin{prop} The map $\mu|_{\HC_1(A)}: \HC_1(A) \to \Ker(\varphi)$ given by $\ll a,b\gg \,\mapsto h(a,b)$ is surjective.
\end{prop}
\begin{proof} Let $x\in \Ker(\varphi)$. Since $\Ker(\varphi)\subset \mcH$, by Lemma \ref{elementinH} we can write $$x = \sum_{i \in I_x}h(a_i,b_i)+ \sum_{j=2}^{m+n} H_{1j}(1,c_j),$$ where $a_i,b_i,c_j \in A$. Then $\varphi(x)=0$ implies that
\begin{equation*}
\sum_{i \in I_x}E_{11}([a_i,b_i]) + \sum_{j=2}^{m+n} \Big(E_{11}(c_j)-(-1)^{|j|(|j|+|c_j|)}E_{jj}(c_j)\Big) =0.
\end{equation*}
Hence $c_j=0$ and $\sum_{i \in I_x}[a_i,b_i]=0$. Therefore, $\sum_{i \in I_x} \ll a_i,b_i\gg\  \in\HC_1(A)$ is a preimage of $x=\sum_{i \in I_x}h(a_i,b_i)$ under the restriction map $\mu|_{\HC_1(A)}$.
\end{proof}

For the next theorem we assume that $A$ has a homogeneous $K$-basis $\{a_\beta\}_{\beta\in\mcB}$
($\mcB$ is an index set), which contains the identity element $1$ of $A$.

\begin{thm}\label{kernel} If $m+n\ge 3$, the kernel of the central extension $(\mfst_{m|n}(A), \varphi)$ of
the Lie superalgebra $\mfsl_{m|n}(A)$ is isomorphic to $\HC_1(A)$ as $K$-modules.
\end{thm}

\begin{proof} We are left to show the injectivity of $\mu|_{\HC_1(A)}$. Let $\mcH_{1j}$ be the $K$-submodule of $\mcH$ generated by $H_{1j}(a,b)$ for $a,b\in A$. Since the set $\{H_{1j}(a_{\beta_1},a_{\beta_2})\}_{\beta_1, \beta_2\in\mcB}$ forms a generating set for $\mcH_{1j}$, we can choose a $K$-basis for $\mcH_{1j}$ from this generating set and denote it by $\{H_{1j}(a_{\beta_1},a_{\beta_2})\}_{\beta_1\in\mcB_{1}, \beta_2\in\mcB_{2}}$, where $\mcB_{1}$ and $\mcB_{2}$ are subsets of $\mcB$. Define a $K$-module homomorphism
$$\nu:\ \mcH_{1j}  \to \; \ll A, A\gg, \, \quad   H_{1j}(a_{\beta_1},a_{\beta_2})\mapsto \,\ll a_{\beta_1},a_{\beta_2}\gg.$$ Then $H_{1j}(a,b)=\,\ll a,b\gg$ for $a,b\in A$.

To show $\mu|_{\HC_1(A)}$ is injective, we let $\mu(\ll a, b\gg)=h(a,b)=H_{1j}(a,b)-(-1)^{|a||b|}H_{1j}(1, ba)=0$. Then $\nu(H_{1j}(a,b))=\nu((-1)^{|a||b|}H_{1j}(1, ba))$, i.e., $\ll a,b\gg\,=(-1)^{|a||b|}\ll1,ba\gg$. In $\ll A, A\gg$ we have $(-1)^{|b||1|}\ll b, a\cdot 1\gg+(-1)^{|a||b|}\ll a, 1\cdot b\gg+(-1)^{|1||a|}\ll 1, ba\gg=0$, so $\ll 1, ba\gg=0$. Thus $\ll a, b\gg=0$ and $\mu|_{\HC_1(A)}$ is injective.
\end{proof}

\begin{rem} In fact, if we define $\nu: \; \mfst_{m|n}(A)\to\; \ll A, A\gg$ by $\nu(F_{ij}(a))=0$ for $i\neq j$, $a\in A$, and $\nu(H_{ij}(a,b))=(-1)^{|i|(|i|+|ab|)}\ll a, b \gg$ for $i\neq j$, $a,b\in A$, then in a similar way as in the proof of Theorem \ref{kernel} we can show that $\nu$ is a well-defined $K$-module homomorphism. In addition, the following diagram is commutative.
$$\xymatrix{\mfst_{m|n}(A)\ar[r]^{\varphi}\ar[d]_{\nu}&\mfsl_{m|n}(A)\ar[d]^{\mathrm{str}}\\
\ll A, A \gg\ar[r]^{d}& [A,A]}$$
\end{rem}

\section{Universal central extension of $\mfsl_{m|n}(A)$, $m+n \geq 5$}

In this section, we will discuss the universality of the central extension $(\mfst_{m|n}(A), \varphi)$ of the Lie superalgebra $\mfsl_{m|n}(A)$ when $m+n \geq 5$.
Let $(\mfe,\psi)$ be an arbitrary central extension of the Lie superalgebra $\mfsl_{m|n}(A)$.
Since $\psi$ is surjective, for any $E_{ij}(a) \in \mfsl_{m|n}(A)$ we can choose some $\wt{E}_{ij}(a) \in \psi^{-1}(E_{ij}(a))$.
First, we observe that the commutator $[\wt{E}_{ij}(a),\wt{E}_{kl}(b)]$ doesn't depend on
the choice of representatives in $ \psi^{-1}(E_{ij}(a)) $ and $\psi^{-1}(E_{kl}(b))$.
This follows from the fact that, if $ \check{E}_{ij}(a) \in \psi^{-1}(E_{ij}(a)) $,
then $\check{E}_{ij}(a) -\wt{E}_{ij}(a) \in \Ker(\psi)$ and $\Ker(\psi)$ is central in $\mfe$.
Next, we define the elements $w_{ij}(a) := [\wt{E}_{ik}(a),\wt{E}_{kj}(1)]$ in $\mfe$ for some $k\neq i,j$.
The following lemma shows that the definition of $w_{ij}(a)$ does not depend on the choice of $k$.

\begin{lem}\label{defeta3}
Let $1\le i \neq j \le m+n$, and suppose that $1\le k,l\le m+n$, and $k,l$ are distinct from $i$ and
$j$. Then we have $ [\wt{E}_{ik}(a),\wt{E}_{kj}(b)] =
[\wt{E}_{il}(a),\wt{E}_{lj}(b)] $.
\end{lem}
\begin{proof}
If $k\neq l$, then we have
\begin{equation*}
\begin{split}
[\wt{E}_{ik}(a),\wt{E}_{kj}(b)] = & \big[ \wt{E}_{ik}(a), [\wt{E}_{kl}(1),\wt{E}_{lj}(b)] + c_1\big]  \\
= & \big[ [\wt{E}_{ik}(a),\wt{E}_{kl}(1)],\wt{E}_{lj}(b) \big] +  (-1)^{(|i|+|k|+|a|)(|k|+|l|)}\big[ \wt{E}_{kl}(1), [ \wt{E}_{ik}(a), \wt{E}_{lj}(b)]\big] \\
= & [\wt{E}_{il}(a) + c_2,\wt{E}_{lj}(b)] + (-1)^{(|i|+|k|+|a|)(|k|+|l|)} [\wt{E}_{kl}(1), c_3] \\
= & [\wt{E}_{il}(a),\wt{E}_{lj}(b)],
\end{split}
\end{equation*}
where $c_1,c_2,c_3$ are central elements in $\mfe$.
\end{proof}

\begin{thm}\label{ext5}
For $m+n\ge 5$, $(\mfst_{m|n}(A),\varphi)$ is the universal
central extension of $\mfsl_{m|n}(A)$.
\end{thm}
\begin{proof} Define a map $\eta: \mfst_{m|n}(A) \rightarrow
\mfe$ by letting
$\eta(F_{ij}(a)) =w_{ij}(a)$. To show that $\eta$ is a well-defined Lie superalgebra homomorphism, we need to
prove the following identities in $\mfe$.
\begin{align}
& w_{ij}(xa+yb) =x w_{ij}(a) +y w_{ij}(b), \; \text{ for all } a,b \in A, \, x,y \in
K. \label{w1}\\
& [w_{ij}(a), w_{jk}(b)] = w_{ik}(ab), \text{ for distinct } i, j, k,
\label{w2}
\\
& [w_{ij}(a), w_{kl}(b)] = 0, \text{ for } i\neq j\neq k\neq l\neq i. \label{w3}
\end{align}

The identity \eqref{w1} follows from the fact that
$\wt{E}_{il}(xa+yb) = x\wt{E}_{il}(a) + y \wt{E}_{il}(b) + c$ with $c
\in \Ker(\psi)$. As for \eqref{w2}, choose $l\neq i,j,k$; then
\begin{equation*}
\begin{split}
[w_{ij}(a), w_{jk}(b)] = & [ w_{ij}(a), [\wt{E}_{jl}(b), \wt{E}_{lk}(1)]] \\
= &  \big[ [w_{ij}(a), \wt{E}_{jl}(b)], \wt{E}_{lk}(1) \big]
+ (-1)^{(|i|+|j|+|a|)(|j|+|l|+|b|)} \big[ \wt{E}_{jl}(b), [w_{ij}(a), \wt{E}_{lk}(1)]\big] \\
= & [\wt{E}_{il}(ab)+c_1, \wt{E}_{lk}(1)] + (-1)^{(|i|+|j|+|a|)(|j|+|l|+|b|)}[\wt{E}_{jl}(b), c_2]  \\
= & w_{ik}(ab), \text{ where }c_1,c_2 \in \Ker(\psi).
\end{split}
\end{equation*}

We need the assumption $m+n\ge 5$ to prove the identity \eqref{w3}
holds. Choose $t\neq i,j,k,l$. Then
\begin{equation*}
\begin{split}
[w_{ij}(a), w_{kl}(b)] = & \big[ w_{ij}(a), [\wt{E}_{kt}(b),\wt{E}_{tl}(1)]\big] \\
= &  \big[ [w_{ij}(a),\wt{E}_{kt}(b)],\wt{E}_{tl}(1)\big] +
(-1)^{(|k|+|t|+|b|)(|i|+|j|+|a|)} \big[ \wt{E}_{kt}(b),  [w_{ij}(a), \wt{E}_{tl}(1)] \big] \\
= & [c_1, \wt{E}_{tl}(1)] + (-1)^{(|k|+|t|+|b|)(|i|+|j|+|a|)} [\wt{E}_{kt}(b), c_2] \\
= & 0, \text{ where }c_1,c_2 \in \Ker(\psi).
\end{split}
\end{equation*}

Thus $\eta$ is a well-defined Lie superalgebra homomorphism. The uniqueness of $\eta$ follows from the fact that, since $F_{ij}(a)
= [F_{ik}(a),F_{kj}(1)] $ for any distinct $i,j,k$, we must have
$\eta(F_{ij}(a)) = [\eta(F_{ik}(a)),\eta(F_{kj}(1))] $ and
$\eta(F_{ik}(a)) - \wt{E}_{ik}(a), \eta(F_{kj}(1)) - \wt{E}_{kj}(1)
\in \Ker(\psi)$.
\end{proof}

We now turn our attention to the Lie superalgebra $\mfsl_I(A)$ (see \cite{NS}), where $I=I_{\bar{0}}\cup I_{\bar{1}}$ is a (possibly infinite) superset. When $|I_{\bar{0}}|=m$ and $|I_{\bar{1}}|=n$, $\mfsl_I(A)$ is the Lie superalgebra $\mfsl_{m|n}(A)$. Let $\mcF$ be the set of finite subsets of $I$ ordered by inclusion. Then the Lie superalgebra $\mfsl_I(A)$ is a direct limit:
$$
\mfsl_I(A)=\bigcup_{F\in\mcF} \mfsl_F(A)\cong \varinjlim_{F\in\mcF}\mfsl_F(A).
$$
Define the Steinberg Lie superalgebra $\mfst_I(A)$ to be the Lie superalgebra over $K$
presented by generators $F_{ij}(a)$ with $i,j\in I,\ i\neq j,\ a\in A$ and relations (\ref{st1})--(\ref{st3}).
We then have a canonical Lie superalgebra epimorphism
$$
\varphi_I: \mfst_I(A)\to \mfsl_I(A), \ \quad F_{ij}(a)\mapsto E_{ij}(a).
$$

\begin{thm}
Assume $|I|\geq 5$. Then $(\varphi_I,\mfst_I(A))$ is the universal central extension of $\mfsl_I(A)$.
In addition, if $A$ is a free $K$-module, then $\Ker(\varphi_I)$ is isomorphic to $\HC_1(A)$.
\end{thm}

\begin{proof} In \cite{NS} Theorem 1.6 it was proved that the $\uce$ functor commutes with the $\dirlim$ functor. Thus we have $$\uce(\mfsl_I(A))\cong\uce(\dirlim_{F\in \mcF}\mfsl_F(A))\cong\dirlim_{F\in \mcF}\uce(\mfsl_F(A)).$$ By Theorem \ref{ext5}, we know that $\uce(\mfsl_F(A))=\mfst_F(A)$ for $|F|\geq 5$. Therefore $$\uce(\mfsl_I(A))\cong \dirlim_{F\in \mcF}\mfst_F(A)\cong \mfst_I(A).$$ In addition, if $A$ is a free $K$-module, we apply Theorem \ref{kernel} and conclude that $\Ker(\varphi_I)\cong\HC_1(A)$.
\end{proof}

\section{Universal central extensions of $\mfsl_{m|n}(A)$, $m+n=3, 4$}

In this section, we will study the universal central extensions of $\mfsl_{m|n}(A)$ when $m+n=3, 4$. By Proposition \ref{inH},
$(\mfst_{m|n}(A),\varphi)$ is a central extension of $\mfsl_{m|n}(A)$ for $m+ n \geq 3$, thus the universal central extension of the Lie superalgebra
$\mfst_{m|n}(A)$ is also the universal central extension of $\mfsl_{m|n}(A)$.

\subsection{Universal central extensions of $\mfst_{2|1}(A)$ and $\mfst_{3|1}(A)$}

The following theorem claims that $\mfst_{2|1}(A)$ and $\mfst_{3|1}(A)$ are both centrally closed, thus the universal central extension of $\mfsl_{m|1}(A)$ is $\mfst_{m|1}(A)$ when $m=2$ or $m=3$.

\begin{thm} \label{2131}
$\mfst_{2|1}(A)$ and $\mfst_{3|1}(A)$ are both centrally closed.
\end{thm}

\begin{proof}
Suppose that
$$
\xymatrix@C=01cm{
0 \ar[r] & \mcV \ar[r] & \mfe \ar[r]^-{\psi} & \mfst_{m|1}(A) \ar[r] & 0}
$$
is a central extension of $\mfst_{m|1}(A)$ for $m=2$ or $m=3$.
We need to show that there exists a unique Lie superalgebra homomorphism
$\eta:\mfst_{m|1}(A) \rightarrow \mfe$ so that $\psi\circ\eta =\id_{\mfst_{m|1}(A)}$.

First, for any $F_{ij}(a) \in \mfst_{m|1}(A)$, we can choose some $\wt{F}_{ij}(a) \in \psi^{-1}(F_{ij}(a))$
and $[\wt{F}_{ij}(a),\wt{F}_{kl}(b)]$ doesn't depend on
the choice of representatives in $ \psi^{-1}(F_{ij}(a)) $ and $\psi^{-1}(F_{kl}(b))$.
Let $\wt{H}_{ij}(a,b)=[\wt{F}_{ij}(a),\wt{F}_{ji}(b)]$,
then
$$
[\wt{H}_{ik}(1,1),\wt{F}_{ij}(a)]-\wt{F}_{ij}(a) \in \Ker(\psi),
$$
where $i,j,k$ are distinct.
Replacing $\wt{F}_{ij}(a)$ by $[\wt{H}_{ik}(1,1),\wt{F}_{ij}(a)]$, then the elements
$\wt{F}_{ij}(a)$ satisfy the relation (1). Moreover, we have $[\wt{H}_{ik}(1,1),\wt{F}_{ij}(a)]=\wt{F}_{ij}(a)$.
By the super Jacobi identity, we have
$$
\left[\wt{H}_{ik}(1,1),[\wt{F}_{ik}(a),\wt{F}_{kj}(b)]\right]=\left[[\wt{H}_{ik}(1,1),\wt{F}_{ik}(a)],\wt{F}_{kj}(b)\right]+
\left[\wt{F}_{ik}(a),[\wt{H}_{ik}(1,1),\wt{F}_{kj}(b)]\right],
$$
which implies
\begin{align*}
[\wt{H}_{ik}(1,1),\wt{F}_{ij}(ab)]&=[\wt{F}_{ik}(a+(-1)^{|i|+|k|}a),\wt{F}_{kj}(b)]-(-1)^{|i|+|k|}[\wt{F}_{ik}(a),\wt{F}_{kj}(b)]\\
&=[\wt{F}_{ik}(a),\wt{F}_{kj}(b)].
\end{align*}
Thus
\begin{equation}
\wt{F}_{ij}(ab)=[\wt{F}_{ik}(a),\wt{F}_{kj}(b)].
\end{equation}
Now, for $k\neq i \neq j \neq k$,
we have
\begin{equation*}
[\wt{F}_{ij}(a),\wt{F}_{ij}(b)]=\left[\wt{F}_{ij}(a),[\wt{F}_{ik}(b),\wt{F}_{kj}(1)]\right]=0,
\end{equation*}
which follows from the super Jacobi identity and $[\wt{F}_{ij}(a),\wt{F}_{ik}(b)], [\wt{F}_{ij}(a),\wt{F}_{kj}(1)] \in \Ker(\psi)$.

Next, we will show that 
$[\wt{F}_{ij}(a),\wt{F}_{kl}(b)]=0$
for $i\neq j\neq k \neq l \neq i$ and $(i,j) \neq (k,l)$. We will consider the following three cases.

\setcounter{case}{0}
\begin{case}
If $|i|=|j|=|k|=|l|$, then we have $m=3$ and $\{i,j,k,l\}=\{1,2,3\}$.
\end{case}
If $i \neq k$, then
\begin{equation*}
\begin{split}
0&=\left[\wt{H}_{i4}(1,1),[\wt{F}_{ij}(a),\wt{F}_{kl}(b)]\right]=\left[[\wt{H}_{i4}(1,1),\wt{F}_{ij}(a)],\wt{F}_{kl}(b)\right]
+\left[\wt{F}_{ij}(a),[\wt{H}_{i4}(1,1),\wt{F}_{kl}(b)]\right]   \\
&=[\wt{F}_{ij}(a),\wt{F}_{kl}(b)]+0=[\wt{F}_{ij}(a),\wt{F}_{kl}(b)].
\end{split}
\end{equation*}
Similarly, we can prove $[\wt{F}_{ij}(a),\wt{F}_{kl}(b)]=0$ if $j \neq l$.

\begin{case}
If $\{i,j,k,l\}=\{1,2,3,4\}$, then we can assume $|i|+|j|=0$.
\end{case}
Now we have
\begin{equation*}
\begin{split}
0&=\left[\wt{H}_{ij}(1,1),[\wt{F}_{ij}(a),\wt{F}_{kl}(b)]\right]=\left[[\wt{H}_{ij}(1,1),\wt{F}_{ij}(a)],\wt{F}_{kl}(b)\right]
+\left[\wt{F}_{ij}(a),[\wt{H}_{ij}(1,1),\wt{F}_{kl}(b)]\right]   \\
&=2[\wt{F}_{ij}(a),\wt{F}_{kl}(b)]+0=2[\wt{F}_{ij}(a),\wt{F}_{kl}(b)].
\end{split}
\end{equation*}
Since $2$ is invertible in $K$, we obtain $[\wt{F}_{ij}(a),\wt{F}_{kl}(b)]=0$.

\begin{case}
Otherwise, we can assume $|i|+|j|=1$ and $i=k$ or $j=l$.
\end{case}

If $i=k$, then
\begin{equation*}
\begin{split}
0&=\left[\wt{H}_{ij}(1,1),[\wt{F}_{ij}(a),\wt{F}_{il}(b)]\right]=\left[[\wt{H}_{ij}(1,1),\wt{F}_{ij}(a)],\wt{F}_{il}(b)\right]
+\left[\wt{F}_{ij}(a),[\wt{H}_{ij}(1,1),\wt{F}_{il}(b)]\right]   \\
&=0+[\wt{F}_{ij}(a),\wt{F}_{il}(b)]=[\wt{F}_{ij}(a),\wt{F}_{il}(b)].
\end{split}
\end{equation*}
Similarly, we can prove $[\wt{F}_{ij}(a),\wt{F}_{kj}(b)]=0$.

The above results imply that
\begin{equation}
[\wt{F}_{ij}(a), \wt{F}_{kl}(b)] = 0, \text{ for } i\neq j\neq k \neq l \neq i,  a, b\in A.
\end{equation}
Now since that $\wt{F}_{ij}(a)$ satisfy the relations \eqref{st1}-\eqref{st3}, there exists a unique Lie superalgebra homomorphism
$$
\eta:\mfst_{m|1}(A) \rightarrow \mfe
$$
such that $\eta(F_{ij}(a))=\wt{F}_{ij}(a)$.
Evidently, $\psi\circ\eta=\id_{\mfst_{m|1}(A)}$ which implies that the original sequence splits.
So $\mfst_{m|1}(A)$ is centrally closed when $m=2$ or $m=3$.
\end{proof}

\subsection{Universal central extensions of $\mfst_{2|2}(A)$}

In this subsection, we will construct explicitly the universal central extension of $\mfst_{2|2}(A)$. Let $\mcI$ be the 2-sided $\Z_2$-graded ideal of $A$ generated by the elements:
$ab-(-1)^{|a||b|}ba$, for homogeneous elements $a,b \in A$.
Let $ A_0:=A/\mcI$ be the quotient superalgebra over $K$. Then $A_0$ is super commutative.
Write $\bar a=a+\mcI$ for $a\in A$.
We define
\begin{equation*}
P_1 = \{(3,1,4,2),\; (3,2,4,1), \; (4,1,3,2),\; (4,2,3,1)\},
\end{equation*}
and
\begin{equation*}
P_2 = \{(1,3,2,4),\; (1,4,2,3), \; (2,3,1,4),\; (2,4,1,3)\}.
\end{equation*}

For $(i, j, k, l) \in P_1\bigsqcup P_2$, let $\epsilon_{ijkl}(A_0)$ denote a copy of $A_0$. We will identify $\epsilon_{ijkl}(\bar{a})$, $\epsilon_{ilkj}(\bar{a})$, $\epsilon_{kjil}(\bar{a})$ and $\epsilon_{klij}(\bar{a})$ for $\bar{a}\in A_0$.
Thus, we have two distinct copies of $A_0$ whose direct sum is denoted by $\mcW$.
Using the decomposition \eqref{direct} of $\mfst_{2|2}(A)$,
we define a $K$-bilinear map $$\psi:\mfst_{2|2}(A)\times\mfst_{2|2}(A) \to \mcW$$ by
$\psi(F_{ij}(a),F_{kl}(b))=(-1)^{j+k+|b|}\epsilon_{ijkl}(\overline{ab})$ for $(i, j, k, l) \in P_1 \bigsqcup P_2$ 
and homogeneous elements $a,b\in A$,
and by $\psi(x,y)=0$ for all other pairs of elements from the summands of \eqref{direct}.
The following lemma shows that the above map $\psi$ is a super $2$-cocycle, thus using $\psi$ we can construct a central extension of $\mfst_{2|2}(A)$.

\begin{lem} \label{cocy2}
The bilinear map $\psi$ is a super $2$-cocycle.
\end{lem}
\begin{proof}
First, we have
\begin{equation*}
\begin{split}
\psi(F_{kl}(b),F_{ij}(a)) = & (-1)^{i+l+|a|} \epsilon_{klij}(\overline{ba}) 
= (-1)^{k+j+|a|+|a||b|} \epsilon_{ijkl}(\overline{ab})  \\
= & - (-1)^{1+|a|+|b|+|a||b|} (-1)^{k+j+|b|} \epsilon_{ijkl}(\overline{ab}) \\
= & - (-1)^{(|i|+|j|+|a|)(|k|+|l|+|b|)} (-1)^{k+j+|b|} \epsilon_{ijkl}(\overline{ab}) \\
= & - (-1)^{(\deg F_{ij}(a))(\deg F_{kl}(b))} \psi(F_{ij}(a),F_{kl}(b)), 
\end{split}
\end{equation*}
where $(i,j,k,l) \in P_1 \bigsqcup P_2$, $a,b \in A$ are homogeneous 
and we use the fact that $|i|+|j|=|k|+|l|=1$. Thus $\psi$ is (super) skew-symmetric.

Next, we will show $J(x,y,z)=0$, where
$$
J(x,y,z)=(-1)^{\deg(x)\deg(z)}\psi([x,y], z) + (-1)^{\deg(x)\deg(y)}\psi([y, z], x) + (-1)^{\deg(y)\deg(z)}\psi([z, x], y)
$$ 
for the homogenous elements $x,y,z \in \mfst_{2|2}(A)$ in summands of \eqref{direct}.
If a term of $J(x,y,z)$ is not $0$, we can assume that $z=F_{pq}(d)$ and $0 \neq [x,y] \in F_{st}(A)$
with $(s,t,p,q) \in P_1 \bigsqcup P_2$. Then we have
\begin{equation*}
\psi([H_{st}(a,b),F_{st}(c)],F_{pq}(d)) = \psi(F_{st}(abc+(-1)^{|s|+|t|+|a||b|+|b||c|+|c||a|}cba),F_{pq}(d)) =0,
\end{equation*}
since $|s|+|t|=1$ and $\overline{abc} = (-1)^{|a||b|+|b||c|+|c||a|}\overline{cba}$.

\setcounter{case}{0}
\begin{case}
If $x$ or $y$ is in $\mcH$, then we can assume $x = H_{ij}(a,b)$, $y=F_{st}(c)$ and $\{i,j\} \cap \{s,t\} \neq \emptyset$.
\end{case}
We only consider one of the four possible cases: $i=s$ and $j=p$. Then we have
\begin{equation*}
\begin{split}
(-1)^{\deg(x)\deg(z)}J(x,y,z) = & \psi([x,y], z) - (-1)^{\deg(y)\deg(z)}\psi([x,z], y) \\
= & \psi([H_{sp}(a,b),F_{st}(c)], F_{pq}(d)) \\
& \quad - (-1)^{(1+|c|)(1+|d|)}\psi([H_{sp}(a,b),F_{pq}(d)], F_{st}(c)) \\
= & \psi(F_{st}(abc), F_{pq}(d)) + (-1)^{(1+|c|)(1+|d|)+|a||b|} \psi(F_{pq}(bad), F_{st}(c))\\
= & (-1)^{t+p+|d|}\epsilon_{stpq}(\overline{abcd}) + (-1)^{(1+|c|)(1+|d|)+|a||b|+q+s+|c|} \epsilon_{pqst}(\overline{badc}) \\
= & (-1)^{t+p+|d|}\epsilon_{stpq}(\overline{abcd}) + (-1)^{1+|d|+q+s} \epsilon_{pqst}(\overline{abcd}) \\
= & 0.
\end{split}
\end{equation*}

\begin{case}
If neither $x$ nor $y$ is in $\mcH$, then we can assume that $x = F_{sp}(a)$, $y=F_{pt}(b)$ or $x = F_{sq}(a)$, $y=F_{qt}(b)$.
\end{case}

We only consider the case when $x = F_{sp}(a)$ and $y=F_{pt}(b)$. Then we have
\begin{equation*}
\begin{split}
(-1)^{\deg(x)\deg(z)}J(x,y,z) = & \psi([x,y], z) - (-1)^{\deg(y)\deg(z)}\psi([x,z], y) \\
= & \psi([ F_{sp}(a),F_{pt}(b)], F_{pq}(d)) - (-1)^{(1+|b|)(1+|d|)}\psi([F_{sp}(a),F_{pq}(d)], F_{pt}(b)) \\
= & \psi(F_{st}(ab), F_{pq}(d)) - (-1)^{(1+|b|)(1+|d|)} \psi(F_{sq}(ad), F_{pt}(b))\\
= & (-1)^{t+p+|d|}\epsilon_{stpq}(\overline{abd}) - (-1)^{(1+|b|)(1+|d|)+q+p+|b|} \epsilon_{sqpt}(\overline{adb}) \\
= & (-1)^{t+p+|d|}\epsilon_{stpq}(\overline{abcd}) - (-1)^{|d|+q+p} \epsilon_{pqst}(\overline{abcd}) \\
= & 0,
\end{split}
\end{equation*}
where we use $\{t,q\} = \{1,2\}$ or $\{3,4\}$ for $(s,t,p,q) \in P_1 \bigsqcup P_2$.

Similarly, we can show that $J(x,y,z)=0$ if $x = F_{sq}(a)$ and $y=F_{qt}(b)$.
\end{proof}

We therefore obtain a central extension of the Lie superalgebra $\mfst_{2|2}(A)$,
$$
0\longrightarrow{\mathcal W}\longrightarrow \mfst_{2|2}(A) \oplus {\mathcal W} \overset{\pi}\longrightarrow\mfst_{2|2}(A)\longrightarrow 0
$$
with Lie superbracket $[(x,c),(y,c')]=([x,y],\psi(x,y))$ for all $x,y \in \mfst_{2|2}(A)$ and $c,c'\in{\mathcal W}$.
$\pi$ is the projection on the first summand.
We denote $\mfst_{2|2}(A) \oplus {\mathcal W}$ with the above bracket by $\mfst_{2|2}(A)^\sharp$.

\begin{thm}
The universal central extension of $\mfst_{2|2}(A)$ is $(\mfst_{2|2}(A)^\sharp, \pi)$.
\end{thm}

\begin{proof}
Suppose that
\begin{equation}
0\longrightarrow{\mathcal V}\longrightarrow \mfg\overset{\tau}{\longrightarrow}\mfst_{2|2}(A)\longrightarrow\notag 0
\end{equation}
is a central extension of $\mfst_{2|2}(A)$.
We need to show that there exists a unique Lie superalgebra homomorphism
$\eta:\mfst_{2|2}(A)^\sharp \longrightarrow \mfg$ such that $\tau\circ\eta=\pi$.

Similarly as what we did in Theorem \ref{2131},
for any $F_{ij}(a) \in \mfst_{2|2}(A)$ we can choose some preimage
$\wt{F}_{ij}(a) \in \tau^{-1}(F_{ij}(a))$ such that they  satisfy the relations (1) and
$$
[\wt{H}_{ik}(1,1),\wt{F}_{ij}(a)]=\wt{F}_{ij}(a),
$$
where $\wt{H}_{ij}(a,b)=[\wt{F}_{ij}(a),\wt{F}_{ji}(b)]$ and $i,j,k$ are distinct.
Moreover, we can get
\begin{equation*}
\wt{F}_{ij}(ab)=[\wt{F}_{ik}(a),\wt{F}_{kj}(b)],
\end{equation*}
and
\begin{equation*}
[\wt{F}_{ij}(a),\wt{F}_{ij}(b)]=[\wt{F}_{ij}(a),\wt{F}_{ik}(b)]=[\wt{F}_{ij}(a),\wt{F}_{kj}(b)]=0,
\end{equation*}
if $i,j,k$ are distinct.
By the choice of preimage $\wt{F}_{ij}(a)$,
we have $\mu_{ijkl}(a,b)=[\wt{F}_{ij}(a),\wt{F}_{kl}(b)] \in \Ker(\tau)$
for $i\neq j\neq k \neq l \neq i$. We first have
\begin{equation*}
\begin{split}
\mu_{ilkj}(bc,a) = & [\wt{F}_{il}(bc),\wt{F}_{kj}(a)] =  [[\wt{F}_{ik}(b),\wt{F}_{kl}(c)],\wt{F}_{kj}(a)]  \\
= & (-1)^{(|k|+|j|+|a|)(|k|+|l|+|c|)}  [[\wt{F}_{ik}(b),\wt{F}_{kj}(a)],\wt{F}_{kl}(c)]  \\
= & (-1)^{(|k|+|j|+|a|)(|k|+|l|+|c|)}  [\wt{F}_{ij}(ba),\wt{F}_{kl}(c)]  \\
= & (-1)^{(|k|+|j|+|a|)(|k|+|l|+|c|)}  \mu_{ijkl}(ba,c).
\end{split}
\end{equation*}
Then taking $b=1$ or $c=1$,
\begin{equation}  \label{murel1}
\mu_{ilkj}(b,a)= (-1)^{(|k|+|j|+|a|)(|k|+|l|+|b|)} \mu_{ijkl}(a,b) = (-1)^{(|k|+|j|+|a|)(|k|+|l|)}  \mu_{ijkl}(ba,1)
\end{equation}
and
\begin{equation} \label{murel}
 \mu_{ijkl}(a,b) =(-1)^{(|k|+|j|+|a|)|b|} \mu_{ijkl}(ba,1)
\end{equation}
where $a,b\in A$ are homogeneous and $i,j,k,l$ are distinct.
On the other hand, similarly as \eqref{HF2} in Lemma \ref{HN}, we have
\begin{equation*}
[\wt{H}_{ij}(a, b), \wt{F}_{ij}(c)]= \wt{F}_{ij}\Big(abc +(-1)^{(|i|+|j|+|a||b|+|b||c|+|c||a|)}cba\Big),
\end{equation*}
hence if $|i|+|j|=0$ we can get
\begin{equation*}
2\mu_{ijkl}(a,b) = [\wt{F}_{ij}(2a),\wt{F}_{kl}(b)] =  [[\wt{H}_{ij}(1,1), \wt{F}_{ij}(a)],\wt{F}_{kl}(b)]  =0,
\end{equation*}
thus $\mu_{ilkj}(b,a)=\mu_{ijkl}(a,b)=0$.
Now if $\{i,j,k,l\} = \{1,2,3,4\}$ and $(i,j,k,l) \notin P_1 \bigsqcup P_2$, then we have $\mu_{ijkl}(a,b)=0$.

On the other hand, for all $(i,j,k,l)\in P_1 \bigsqcup P_2$,
$|i|+|j|=|k|+|j|=\bar{1}$, then for any homogeneous elements $a,b,c \in A$, \eqref{murel} implies
\begin{equation*}
\begin{split}
\mu_{ijkl}\big(c(ab-(-1)^{|a|\cdot|b|}ba),1\big)
= & (-1)^{(|k|+|j|+|a|+|b|)|c|} \mu_{ijkl}(ab-(-1)^{|a|\cdot|b|}ba,c)  \\
= & (-1)^{(|k|+|j|+|a|+|b|)|c|} \mu_{ijkl}(ab+(-1)^{|i|+|j|+|a|\cdot|b|}ba,c) \\
= & (-1)^{(|k|+|j|+|a|+|b|)|c|}[\wt{F}_{ij}(ab+(-1)^{|i|+|j|+|a|\cdot|b|}ba),\wt{F}_{kl}(c)] \\
= & (-1)^{(|k|+|j|+|a|+|b|)|c|}\Big[[\wt{H}_{ij}(a,b),\wt{F}_{ij}(1)],\wt{F}_{kl}(c)\Big] \\
= & 0,
\end{split}
\end{equation*}
which shows $\mu_{ijkl}({\mathcal I},1)=0$ for $(i,j,k,l)\in P_1\bigsqcup P_2$.
In particular, we have
\begin{equation*}
\mu_{ijkl}(a,b) = (-1)^{(1+|a|)|b|} \mu_{ijkl}(ba,1) = (-1)^{|b|} \mu_{ijkl}(ab,1).
\end{equation*}

In \eqref{murel1}, taking $a=1$ we have $\mu_{ilkj}(a,1) = -\mu_{ijkl}(a,1)$ and $\mu_{ijkl}(a,1) = (-1)^{|a|}\mu_{ijkl}(1,a)$ 
for $(i,j,k,l)\in P_1\bigsqcup P_2$.
In particular, we get
\begin{equation*}
\mu_{1423}(a,1) = -\mu_{1324}(a,1)
\end{equation*}
and
\begin{equation*}
\mu_{1423}(a,1) = -(-1)^{1+|a|}\mu_{2314}(1,a) = \mu_{2314}(a,1) = -\mu_{2413}(a,1).
\end{equation*}
Similarly, we have
\begin{equation*}
\mu_{3241}(a,1) = -\mu_{3142}(a,1)= \mu_{4132}(a,1)=-\mu_{4231}(a,1).
\end{equation*}

Now we define a map $$\eta:\mfst_{2|2}(A)^\sharp \longrightarrow \mfg$$ by
$\eta(F_{pq}(a))=\wt{F}_{pq}(a)$ and $\eta(\epsilon_{ijkl}(\overline{a}))=\mu_{ijkl}(a,1)$ for $1 \leq p \neq q \leq 4$, $a \in A$ and $(i,j,k,l) \in P_1\bigsqcup P_2$.
By the above discussion, we know that $\eta$ is a unique Lie superalgebra homomorphism such that $\tau\circ\eta=\pi$.
\end{proof}

\textit{Acknowledgements.} The authors thank the referee for pointing out the reference \cite{KT} and the suggestion on adding the last section of the paper. The authors thank E. Neher for pointing out the reference \cite{M} and useful comments on an earlier version of the paper.

\noindent Hongjia Chen\\
\noindent School of Mathematical Sciences\\
\noindent University of Science and Technology of China\\
\noindent  Hefei, Anhui 230026, China\\
\noindent \textit{E-mail address}: hjchen@ustc.edu.cn
\\

\noindent Jie Sun\\
\noindent Department of Mathematical Sciences\\
\noindent Michigan Technological University\\
\noindent  Houghton, MI 49931, USA\\
\noindent \textit{E-mail address}: sjie@mtu.edu\\

\end{document}